\documentclass[a4paper,dvips,11pt]{article}
\usepackage[latin1]{inputenc}
\usepackage{amsmath,amsthm,amsfonts,amssymb}
\usepackage{enumerate}
\usepackage{graphicx,psfrag}
\usepackage{pst-node,pstricks-add,pst-text,pst-3d,pst-tree}
\usepackage[ps2pdf,colorlinks]{hyperref}

\newlength{\hdelta}
\newlength{\vdelta}
\theoremstyle{plain}
\newtheorem{theorem}{Theorem}[section]
\newtheorem{lemma}[theorem]{Lemma}

\newtheorem{proposition}[theorem]{Proposition}

\theoremstyle{definition}

\newtheorem{example}[theorem]{Example}

\numberwithin{equation}{section}

\newcommand{\mailto}[1]{\href{mailto:#1}{\nolinkurl{#1}}}

\newcommand{\rc}{\overline{G}}

\newcommand{\E}{\operatorname{E}}
\newcommand{\pr}{\operatorname{P}}
\newcommand{\eqst}{=_{\rm{st}}}

\newcommand{\Z}{\mathbb{Z}}

\newcommand{\R}{\mathbb{R}}

\begin{document}

\title{Juggler's exclusion process}

\author{
 Lasse Leskelä\thanks{
 Postal address:
 Department of Mathematics and Statistics,
 PO Box 35,
 40014 University of Jyväskylä,
 Finland.
 Tel: +358 14 260 2728.
 URL: \url{http://www.iki.fi/lsl/} \quad
 Email: \protect\mailto{lasse.leskela@iki.fi}}
 \and
 Harri Varpanen\thanks{
 Postal address:
 Department of Mathematics and Systems Analysis,
 Aalto University,
 PO Box 11100, 00076 Aalto, Finland.
 Tel: +358 9 4702 3069.
 URL: \url{http://math.tkk.fi/en/people/harri.varpanen} \quad
 Email: \protect\mailto{harri.varpanen@tkk.fi}}
}
\date{\today}

\maketitle

\begin{abstract}
Juggler's exclusion process describes a system of particles on the positive integers where
particles drift down to zero at unit speed. After a particle hits zero, it jumps into a randomly
chosen unoccupied site. We model the system as a set-valued Markov process and show that the
process is ergodic if the family of jump height distributions is uniformly integrable. In a
special case where the particles jump according to a set-avoiding memoryless distribution, the
process reaches its equilibrium in finite nonrandom time, and the equilibrium distribution can be
represented as a Gibbs measure conforming to a linear gravitational potential.

\end{abstract}


\section{Introduction}

Juggler's exclusion process (JEP) describes a system of $n$ indistinguishable particles on the
positive integers $\Z_+ = \{0,1,2,\dots\}$, where particles drift down to zero at unit speed, and
after a particle hits zero, it jumps into a randomly chosen unoccupied site. Denote by $S_n$ the
countable space of $n$-element subsets of $\Z_+$. A JEP is defined as a discrete-time Markov
process in $S_n$ with transition probability matrix
\begin{equation}
 \label{eq:TransitionMatrix}
 P(A,B) = \left\{
 \begin{aligned}
   1,              & \quad \text{if} \ 0 \notin A, \ B = A-1, \\
   \nu_{A^*-1}(y), & \quad \text{if} \ 0 \in A, \ B = (A^*-1) \cup \{y\}, \\
   0,              & \quad \text{else},
 \end{aligned}
 \right.
\end{equation}
where $A-1$ is the set obtained by shifting the elements of $A$ down by one position, $A^* = A
\setminus \{\min A\}$ is the set obtained by deleting the smallest element of $A$, and $\nu =
(\nu_B)$ is a family of probability measures on $\Z_+$ indexed by $B \in S_{n-1}$, satisfying
\begin{equation}
 \label{eq:Noncolliding}
 \sum_{y \in B} \nu_B(y) = 0 \quad \text{for all $B \in S_{n-1}$}.
\end{equation}
In general, a set $A \in S_n$ describes the heights of the particles in the system. The
probability measure $\nu_B$ describes the height of a particle after a jump which occurred while
the other $n-1$ particles drifted down from configuration $B+1$ into configuration $B$.
Assumption~\eqref{eq:Noncolliding} guarantees that all particles jump into unoccupied sites, and
therefore the number of particles in the system remains fixed at each transition.

The name JEP is inspired by viewing the particles as balls that a juggler throws into a discrete
set of heights according to a random pattern. In this context condition~\eqref{eq:Noncolliding}
amounts to a so-called siteswap juggling pattern where only one ball is allowed to be at hand
(zero height) at any time instant. Such patterns have been studied mostly in the deterministic
periodic setting, see e.g. Knutson, Lam and Speyer \cite{Knutson_Lam_Speyer} and references
therein. In addition, Warrington \cite{Warrington_2005} has computed the equilibrium distribution
of a random siteswap pattern where throw heights are bounded and uniformly distributed.

JEPs may be naturally encountered in various application areas such as quantum physics and
engineering. For example, if particle heights are viewed as energy levels, a JEP can be viewed as
a system of $n$ randomly excited particles subject to the condition that no two particles may
share the same energy level. Alternatively, a JEP may represent the residual job completion times
in a manufacturing system of $n$ machines running in parallel, which is operated under the
condition that only one job may complete at any time instant. More generally, the particle heights
in a JEP may be viewed as residual lifetimes in a discrete-time point process which may be
regarded as a superposition of $n$ coupled renewal processes.

A probability measure $\pi$ on $S_n$ is an equilibrium distribution of a JEP if and only if it
satisfies the balance equation $\pi P = \pi$, or equivalently,
\begin{equation}
 \label{eq:Equilibrium}
 \pi(B)
 \ = \ \pi(B+1) \ + \sum_{\substack{A \in S_{n-1} \\ A \subset B}} \pi( (A+1) \cup \{0\} ) \, \nu_{A}(B \setminus A)
\end{equation}
for all $B \in S_n$. The following example illustrates the solution of~\eqref{eq:Equilibrium} for
one-particle systems.

\begin{example}[One-particle JEP]
\label{exa:JEP1}
For $n=1$ the family $(\nu_B)$ reduces to a single probability measure $\nu_\emptyset$. In this
case the height of the unique particle in the system's configuration is a discrete-time Markov
process on $\Z_+$ with transition probabilities
\[
 P(x,y) = \left\{
 \begin{aligned}
   1,                & \quad \text{if} \ x > 0, \ y = x - 1, \\
   \nu_\emptyset(y), & \quad \text{if} \ x = 0, \ y \ge 0, \\
   0,                & \quad \text{else}.
 \end{aligned}
 \right.
\]
This process may be identified as the residual lifetime of a discrete-time renewal process with
interevent distribution $\nu_\emptyset$ (e.g.\ Asmussen~\cite[Section I.2]{Asmussen_2003}). A
direct computation shows that this JEP has an equilibrium distribution if and only if $m = \sum_{x
\ge 0} x \nu_\emptyset(x) < \infty$, in which case the equilibrium is unique and given by $\pi(x)
= m^{-1} \nu_\emptyset([x,\infty))$. We note that $\pi = \nu_\emptyset$ in the special case where
$\nu_\emptyset$ is geometric, a well-known fact in renewal theory.
\end{example}

In this paper we will show that a wide class of JEPs have a unique equilibrium distribution
(Theorem~\ref{the:Stability}). However, finding an analytical formula for the equilibrium by
solving~\eqref{eq:Equilibrium} for a general JEP with $n \ge 2$ particles appears difficult if not
impossible. We saw in Example~\ref{exa:JEP1} that for a geometric jump height distribution the
equilibrium is also geometric. This suggests that a closed-form analytical formula for
many-particle systems might be available for some special jump height distributions. The problem
is that condition~\eqref{eq:Noncolliding} rules out geometric height distributions in the
many-particle case. To get around this problem, we introduce the notion of a memoryless
distribution on a subset of integers, and show that when particles jump according to such
distribution, the equilibrium can be expressed as a Gibbs measure conforming to a linear
gravitational potential (Theorems~\ref{the:GeometricEquilirium} and~\ref{the:GibbsRealization}).
Such JEPs are shown to converge to equilibrium ultrafast: in a finite nonrandom time
(Theorem~\ref{the:Convergence}).

The model studied in this paper may be seen as a special instance of an exclusion process,
although most literature on exclusion processes is focused on continuous-time models with
infinitely many particles, especially the asymmetric simple exclusion process (ASEP) on the doubly
infinite integer lattice (see Grimmett~\cite[Sec.~10.4]{Grimmett_2010} or
Liggett~\cite[Sec.~VIII]{Liggett_1985} for a general overview, and Martin and
Schmidt~\cite{Martin_Schmidt} for discrete-time ASEP models). To avoid confusion, we emphasize
that as an exclusion process, a JEP is very special because the movement of particles is
deterministic, and the only source of randomness are the jumps from the boundary.

The rest of this paper is organized as follows. Section~\ref{sec:Stability} presents a general
theorem for the existence of a unique equilibrium under natural integrability conditions on the
jump height distributions. Section~\ref{sec:MemorylessJEP} provides a detailed analysis of JEPs
generated by memoryless jump height distributions. Section~\ref{sec:RelatedModels} briefly
discusses two models closely related to memoryless JEPs, and Section~\ref{sec:Conclusions}
concludes the paper.

\section{Existence of an equilibrium}
\label{sec:Stability}

The goal of this section is to study when the probability distribution of a JEP converges to a
unique equilibrium distribution, regardless of its initial state. To rule out Markov processes
with multiple recurrent classes and periodic dynamics, the following result gives a simple
condition.

\begin{lemma}
\label{the:aperiodic}
Any JEP generated by a family $(\nu_B)$ such that
\begin{equation}
 \label{eq:aperiodic}
 \nu_B( \min(B^c) )> 0 \quad \text{for all $B \in S_n$}
\end{equation}
is aperiodic and has a unique recurrent class.
\end{lemma}
\begin{proof}
Denote by $G=[0,n-1]$ the ground state of $S_n$ where the particles are located as low as
possible. Let us write $A \to B$ if $P^t(A,B) > 0$ for some $t \ge 1$, and define $\rc = \{A \in
S_n: G \to A\}$ as the set of configurations accessible from $G$. Using~\eqref{eq:aperiodic} it is
not hard to verify that $A \to G$ for any $A \in S_n$, which implies that $\rc$ is the unique
recurrent class of the process. Assumption~\eqref{eq:aperiodic} further yields $P(G,G) > 0$, which
implies that the process is aperiodic.
\end{proof}

Not all aperiodic JEPs converge to an equilibrium. For example, if each particle jumps to a height
twice the maximum height of the other particles, then such a process evidently drifts to infinity.
To rule out such nonphysical examples, we need to impose some conditions on the jump height
distributions. A family of probability measures $(\nu_B)$ on $\Z_+$ is called \emph{uniformly
integrable} if
\begin{equation}
 \label{eq:UI}
 \sup_{B} \sum_{x > K} x \nu_B(x) \to 0 \quad \text{as} \ K \to \infty.
\end{equation}

The following result shows that a wide class of juggler's exclusion processes is stochastically
stable in the sense that as time tends to infinity, the probability distribution of the particle
configuration converges to a unique equilibrium.

\begin{theorem}
\label{the:Stability}
Any JEP generated by a uniformly integrable family of distributions $(\nu_B)$
satisfying~\eqref{eq:aperiodic} has a unique equilibrium distribution $\pi$, and the process
started from any initial configuration converges in total variation to $\pi$ as time tends to
infinity.
\end{theorem}
\begin{proof}
Define the function $V: S_n \to \R_+$ by $V(A) = \max A$, and denote
\[
 PV(A) = \sum_{B \in S_n} V(A) P(A,B),
\]
where $P(A,B)$ are the transition probabilities in~\eqref{eq:TransitionMatrix}. If $0 \notin A$,
then
\begin{equation}
 \label{eq:DriftPositive}
 PV(A) - V(A) = V(A-1) - V(A) = - 1.
\end{equation}
Assume next that $0 \in A$. Then
\[
 PV(A) = \sum_{y \in \Z_+} V((A^*-1) \cup \{y\}) \nu_{A^*-1}(y),
\]
and because $V((A^*-1) \cup \{y\}) = \max(V(A)-1,y)$, it follows that
\begin{align*}
 PV(A)
 & = \sum_{y < V(A)} (V(A)-1) \nu_{A^*-1}(y) + \sum_{y \ge V(A)} y \nu_{A^*-1}(y) \\
 & \le V(A)-1 + \sup_{B \in S_{n-1}} \sum_{y \ge V(A)} y \nu_B(y).
\end{align*}
The uniform integrability of $(\nu_B)$ implies that the last term on the right side above is less
than $\frac12$ whenever $V(A)$ is large enough. After combining this observation
with~\eqref{eq:DriftPositive}, we conclude that there exists a number $K>0$ such that
\begin{equation}
 \label{eq:DriftBig}
 PV(A) - V(A) \le -\frac12 \quad \text{for all $A \in S_n$ such that $V(A) > K$}.
\end{equation}
The uniform integrability of $(\nu_B)$ also implies that $PV(A) < \infty$ for all $A \in S_n$.
Because the set $\{A \in S_n: V(A) \le K\}$ is finite, we conclude that
\begin{equation}
 \label{eq:DriftSmall}
 \sup_{A \in S_n: V(A) \le K} PV(A) < \infty.
\end{equation}

Further, Lemma~\ref{the:aperiodic} implies that the JEP is aperiodic and has a unique recurrent
class $\rc$. As a consequence, the process is $\psi$-recurrent (Meyn and
Tweedie~\cite[Sec.~4.2]{Meyn_Tweedie_1993}), where the measure $\psi$ on the power set of $S_n$ is
defined by $\psi(\overline{A}) = | \rc \cap \overline{A} |$. In light of the Foster--Lyapunov
drift bounds~\eqref{eq:DriftBig} and~\eqref{eq:DriftSmall}, the claim now follows by applying
\cite[Theorem 14.0.1]{Meyn_Tweedie_1993} (with $f$ identically one).
\end{proof}

\section{JEP with memoryless height distributions}
\label{sec:MemorylessJEP}

\subsection{Memoryless distribution on a subset of positive integers}
\label{sec:MemorylessDistribution}

This section is devoted to analyzing JEPs generated by height distributions
\begin{equation}
 \label{eq:MemorylessDistribution}
 \nu_B(y) = \left\{
 \begin{aligned}
   (1-\alpha) \alpha^{h_B(y)}, &\quad y \in B^c, \\
   0,                          &\quad y \in B,
 \end{aligned}
 \right.
\end{equation}
where $h_B(y) = | [0,y-1] \setminus B |$ is the number of points in $B^c$ strictly less than $y$,
and $\alpha \in (0,1)$. The probability measure defined by~\eqref{eq:MemorylessDistribution} shall
be called the \emph{memoryless distribution} on $B^c$ with parameter $\alpha$. This nomenclature
is motivated by the following property, which states that if $\xi_B$ denotes a random variable
with distribution $\nu_B$, then the conditional distribution of $\xi_B$ given that $\xi_B \ge s$
is the same as the distribution of the random variable $s + \xi_{B-s}$.
\begin{proposition}
\label{the:Memoryless}
For all finite $B \subset \Z_+$, the probability measure $\nu_B$ defined
by~\eqref{eq:MemorylessDistribution} satisfies
\begin{equation}
 \label{eq:Memoryless}
 \nu_B[t,\infty) = \nu_B[s,\infty) \, \nu_{B-s}[t-s,\infty) \quad \text{for all $s \le t$}.
\end{equation}
\end{proposition}
\begin{proof}
By applying the identity $h_B(s+t) = h_B(s) + h_{B-s}(t)$, one can check by a direct computation
that $\nu_B[t,\infty) = \alpha^{h_B(t)}$ and $\nu_{B-s}[t-s,\infty) = \alpha^{h_B(t)-h_B(s)}$ for
all positive integers $s$ and $t$ such that $s \le t$.
\end{proof}
Observe that for $B = \emptyset$, formula~\eqref{eq:Memoryless} is the standard memoryless
property on $\Z_+$, and the memoryless distribution $\nu_\emptyset$ on $B^c = \Z_+$ is the
geometric distribution on $\Z_+$ with success probability $\alpha$.

The memoryless distribution on $A^c$ can be represented in terms of a geometrically distributed
random variable by using the following function. We define for a given finite $A \subset \Z_+$ the
\emph{$A$-avoiding shift} by the formula
\begin{equation}
 \label{eq:DefShift}
 \theta_A(x) = \min \{n \ge 0: | \, [0,n] \setminus A| \, \ge \, x+1\}, \quad x \in \Z_+.
\end{equation}
The $A$-avoiding shift is the natural bijection from $\Z_+$ onto the complement of $A$, and
$\theta_A(x)$ equals the $(x+1)$-th element of $A^c$ (see Figure~\ref{fig:Shift}). Some basic
technical facts about this map that are needed in the sequel are listed in
Appendix~\ref{sec:ShiftTechnical}.

\begin{figure}[h]
\small
\begin{center}
  \psset{unit=1.2cm}
  \begin{pspicture}(0,-.1)(5,1.5)

    \psline[linewidth=.8pt]{->}(0,1)(5.5,1)
    \psline[linewidth=.8pt](0,1.1)(0,.9)
    \psline[linewidth=.8pt](1,1.1)(1,.9)
    \psline[linewidth=.8pt](2,1.1)(2,.9)
    \psline[linewidth=.8pt](3,1.1)(3,.9)
    \psline[linewidth=.8pt](4,1.1)(4,.9)
    \psline[linewidth=.8pt](5,1.1)(5,.9)

    \psline[linewidth=.8pt]{->}(0,0)(5.5,0)
    \psline[linewidth=.8pt](1,.1)(1,-.1)
    \psline[linewidth=.8pt](4,.1)(4,-.1)
    \psline[linewidth=.8pt](5,.1)(5,-.1)

    \psdots[dotstyle=o](0,0)(2,0)(3,0)

    \psline{->}(0,1)(1,0)
    \psline{->}(1,1)(4,0)
    \psline{->}(2,1)(5,0)

    \uput[90](0,1){$0$}
    \uput[90](1,1){$1$}
    \uput[90](2,1){$2$}
    \uput[90](3,1){$3$}
    \uput[90](4,1){$4$}
    \uput[90](5,1){$5$}

    \uput[-90](1,0){$\theta_A(0)$}
    \uput[-90](4,0){$\theta_A(1)$}
    \uput[-90](5,0){$\theta_A(2)$}
  \end{pspicture}
\end{center}
\caption{\label{fig:Shift} The action of the $A$-avoiding shift for $A = \{0,2,3\}$.}
\end{figure}
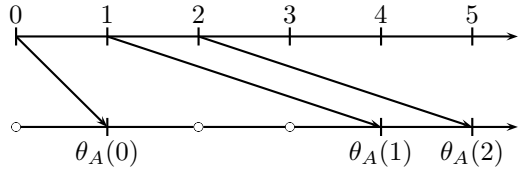

\begin{lemma}
\label{the:ShiftDistribution}
Let $B$ be a finite subset of $\Z_+$, and let $\xi$ be a geometric random variable on $\Z_+$ with
success probability $\alpha \in (0,1)$. Then the random variable $\theta_B(\xi)$ has the
memoryless distribution~\eqref{eq:MemorylessDistribution} on $B^c$.
\end{lemma}
\begin{proof}
Because $\theta_B$ maps $\Z_+$ onto $B^c$, it is clear that $\pr(\theta_B(\xi) \in B) = 0$. The
claim thus follows immediately after noting that $\theta_B^{-1}(x) = h_B(x)$ for $x \in B^c$ by
Lemma~\ref{the:ShiftInverse}.
\end{proof}

\subsection{Noncolliding union of geometric random variables}
\label{sec:Union}

The \emph{noncolliding union} of a finite (possibly empty) set $A \subset \Z_+$ and (possibly
nondistinct) points $x_1,\dots,x_n \in \Z_+$ is denoted by
\begin{equation}
 \label{eq:DefUnion}
 U(A,x_1,\dots,x_n)
\end{equation}
and defined using the $A$-avoiding shift~\eqref{eq:DefShift} recursively by
\begin{align*}
 U(A, x_1) &= A \cup \{ \theta_A(x_1) \}, \\
 U(A,x_1,\dots,x_{k+1}) &= U(A,x_1,\dots,x_k) \cup \left\{ \theta_{U(A,x_1,\dots,x_k)}(x_{k+1}) \right\}.
\end{align*}
The \emph{noncolliding union} of points $x_1,\dots,x_n \in \Z_+$ is defined by substituting $A =
\emptyset$ into~\eqref{eq:DefUnion}, and denoted by
\begin{equation}
 \label{eq:DefUnionEmpty}
 U(x_1,\dots,x_n).
\end{equation}
By construction, $U(A,x_1,\dots,x_n)$ is a set with $|A|+n$ elements which contains $A$ but not
necessarily $x_1,\dots,x_n$. Elementary properties of the noncolliding union are listed in
Appendix~\ref{sec:UnionTechnical}.

Fix a parameter $\alpha \in (0,1)$, and let $\xi_1,\xi_2,\dots$ be independent random variables
all having the geometric distribution on $\Z_+$ with success probability $\alpha$. We will denote
the $n$-element random set obtained as the noncolliding union of $\xi_1,\dots,\xi_n$ by
\begin{equation}
 \label{eq:RandomSetEmpty}
 G_n = U(\xi_1,\dots,\xi_n).
\end{equation}
In addition, the noncolliding union of a finite $B \subset \Z_+$ and $\xi_1,\dots,\xi_n$ will be
denoted by
\begin{equation}
 \label{eq:RandomSet}
 G_n(B) = U(B,\xi_1,\dots,\xi_n).
\end{equation}
The following two lemmas extend well-known facts for geometric random variables to the random sets
$G_n$ and $G_n(B)$.

\begin{lemma}
\label{the:RandomSetPositive}
For any finite (possibly empty) $B \subset \Z_+$ such that $0 \notin B$, the random set $G_n(B)$
defined by~\eqref{eq:RandomSet} satisfies
\[
 \E\{ f(G_n(B)) \, | \, 0 \notin G_n(B) \} = \E f( G_n(B-1) + 1)
\]
for all bounded functions $f$ on $S_{|B|+n}$.
\end{lemma}
\begin{proof}
By Lemma~\ref{the:UnionMin}, $0 \notin G_n(B)$ if and only if $\min \{\xi_1,\dots,\xi_n\} > 0$.
Therefore,
\[
 \E\{ f(G_n(B)) \, | \, 0 \notin G_n(B) \}
 = \E \left\{ f(U(B,\xi_1,\dots,\xi_n)) \, | \, \xi_1 > 0, \dots \xi_n > 0 \right\}.
\]
Because the geometric random variable $\xi_k$, conditioned on being strictly positive, has the
same distribution as $\xi_k + 1$, and because $\xi_1,\dots,\xi_n$ are independent, it follows that
\[
 \E\{ f(G_n(B)) \, | \, 0 \notin G_n(B) \}
 = \E f(U(B,\xi_1+1,\dots,\xi_n+1)),
\]
so the claim follows by Lemma~\ref{the:UnionShift}.
\end{proof}

\begin{lemma}
\label{the:RandomSetZero}
For any finite (possibly empty) $B \subset \Z_+$ such that $0 \notin B$, the random set $G_n(B)$
defined by~\eqref{eq:RandomSet} satisfies
\[
 \E \left\{ f(G_n(B)^*) \, | \, 0 \in G_n(B) \right\}
 = \E f(G_{n-1}(B-1)+1)
\]
for all bounded functions $f$ on $S_{|B|+n-1}$.
\end{lemma}
\begin{proof}
With the interpretation $G_0(B-1) = B-1$, the claim makes sense also for $n=1$. In this case the
claim follows by observing that $G_1(B)^* = B$ on the event that $\xi_1=0$. To proceed by
induction, assume that the claim holds for some $n$. By definition, $G_{n+1}(B) = U(G_n(B),
\xi_{n+1})$, so we can split the event $\{ 0 \in G_{n+1} \} $ into disjoint events
\[
 \Omega_1
 = \{ 0 \notin G_n(B), \ \xi_{n+1} = 0\}
 \quad \text{and} \quad
 \Omega_2 = \{ 0 \in G_n(B) \}.
\]

Let us first analyze the event $\Omega_1$. On this event, $G_{n+1}(B)^* = G_n(B)$, so by
Theorem~\ref{the:RandomSetPositive} and the independence of $G_n(B)$ and $\xi_{n+1}$, we see that
\begin{equation}
\label{eq:Omega1}
\begin{aligned}
 \E \{ f(G_{n+1}(B)^*) \, | \, \Omega_1 \}
 &= \E \{ f(G_n(B)) \, | \, 0 \notin G_n(B) \} \\
 &= \E f(G_n(B-1)+1).
\end{aligned}
\end{equation}

Let us next analyze the event $\Omega_2$. In light of Lemma~\ref{the:UDelete}, we see that on this
event,
\[
 G_{n+1}(B)^* = U(G_n(B),\xi_{n+1})^* = U(G_n(B)^*, \xi_{n+1} + 1).
\]
Because $G_n(B)^*$ and $\xi_{n+1}$ are independent, we obtain by conditioning on $\xi_{n+1}$,
applying the induction assumption to the map $A \mapsto f(U(A,\xi_{n+1}+1))$, and then using
Lemma~\ref{the:UnionShift} that
\begin{align*}
 \E \{ f(G_{n+1}(B)^*) \, | \, \Omega_2 \}
 &= \E \{ f(U(G_n(B)^*, \xi_{n+1} + 1)) \, | \, 0 \in G_n(B) \} \\
 &= \E f(U(G_{n-1}(B-1) + 1, \xi_{n+1} + 1)) \\
 &= \E f(U(G_{n-1}(B-1), \xi_{n+1})+1).
\end{align*}
Note that
\[
 U(G_{n-1}(B-1), \xi_{n+1}) \eqst U(G_{n-1}(B-1),\xi_n) = G_n(B-1),
\]
because $\xi_n$ and $\xi_{n+1}$ have the same distribution and both random variables are
independent of $G_{n-1}(B-1)$. 
(We use the notation $X \eqst Y$ to mean that the random elements
$X$ and $Y$ have the same distribution.)
This implies that
\begin{equation}
 \label{eq:Omega2}
  \E \{ f(G_{n+1}(B)^*) \, | \, \Omega_2 \}
  = \E f(G_n(B-1) + 1).
\end{equation}
By combining~\eqref{eq:Omega1} and \eqref{eq:Omega2}, we conclude that the claim is true for
$n+1$. This completes the induction step, and thus also the proof.
\end{proof}

\subsection{One-step evolution}

Assuming that the height distributions are memoryless as in~\eqref{eq:MemorylessDistribution}, we
see with the help of Lemma~\ref{the:ShiftDistribution} that the JEP may be constructed using the
noncolliding union defined by~\eqref{eq:DefUnion} according to
\begin{equation}
 \label{eq:NextStep}
 X_{t+1} = \left\{
 \begin{aligned}
   X_t - 1,                 & \quad 0 \notin X_t,\\
   U(X_t^* - 1, \xi_{t+1}), & \quad 0 \in X_t,
 \end{aligned}
 \right.
\end{equation}
where $\xi_1,\xi_2,\dots$ are independent geometric random variables in $\Z_+$ with success
probability $\alpha$, independent of the initial state $X_0$.

The following key technical result describes the one-step evolution of the JEP with memoryless
height distributions, when started at a random initial state distributed according to $G_k(B)$
defined in~\eqref{eq:RandomSet}.

\begin{lemma}
\label{the:key}
If $X_0 \eqst G_k(B)$, then the JEP with memoryless height distributions at time $1$ is
distributed according to
\[
 X_{1} \eqst \left\{
  \begin{aligned}
   G_k(B-1),       & \quad 0 \notin B, \\
   G_{k+1}(B^*-1), & \quad 0 \in B.
  \end{aligned}
 \right.
\]
\end{lemma}
\begin{proof}
Assume first that $0 \notin B$. Fix a bounded function $f$. Then by first
using~\eqref{eq:NextStep} and then Lemma~\ref{the:RandomSetPositive}, we find that
\begin{align*}
 \E \left\{ f(X_{1}) \, | \, 0 \notin X_0 \right\}
 &= \E \left\{ f(G_k(B)-1) \, | \, 0 \notin G_k(B) \right\} \\
 &= \E f(G_k(B-1)).
\end{align*}
Further, because in~\eqref{eq:NextStep} the random variables $\xi_{t+1}$ and $X_t^* - 1$ are
independent,
\[
 \E \left\{ f(X_{1}) \, | \, 0 \in X_0 \right\}
 = \E \left\{ f(U(G_k(B)^* - 1, \xi_{k+1})) \, | \, 0 \in G_k(B) \right\},
\]
so by Lemma~\ref{the:RandomSetZero} and independence,
\begin{align*}
 \E \left\{ f(X_{1}) \, | \, 0 \in X_0 \right\}
 &= \E f(U(G_{k-1}(B-1), \xi_{k+1})) \\
 &= \E f(G_k(B-1)).
\end{align*}
By combining these two observations, we conclude that the claim holds when $0 \notin B$.

Let us next analyze the case where $0 \in B$. Then with the help of Lemma~\ref{the:UDelete} we see
that
\begin{align*}
 X_{1}
 &\eqst U( G_k(B)^*-1, \xi_{k+1}) \\
 &= U( U(B^*,\xi_1+1,\dots,\xi_k+1) - 1, \xi_{k+1}) \\
 &= U( U(B^*-1,\xi_1,\dots,\xi_k), \xi_{k+1}) \\
 &= G_{k+1}(B^*-1).
\end{align*}
\end{proof}

\subsection{Ultrafast convergence to equilibrium}

By applying the key result, Lemma~\ref{the:key}, we may identify the equilibrium of the memoryless
JEP.

\begin{theorem}
\label{the:GeometricEquilirium}
The JEP generated by the memoryless height distributions~\eqref{eq:MemorylessDistribution} has a
unique equilibrium which has the same distribution as the noncolliding union $G_n =
U(\xi_1,\dots,\xi_n)$ of $n$ independent geometric random variables defined
in~\eqref{eq:RandomSetEmpty}.
\end{theorem}
\begin{proof}
Assume that the JEP is started at a random initial state distributed according to $G_n$. Then by
Lemma~\ref{the:key} applied with $B = \emptyset$, it follows that $X_1 \eqst G_n$. By the Markov
property, it follows that $X_t \eqst G_n$ for all $t$. The uniqueness of the equilibrium
distribution follows from Theorem~\ref{the:Stability}.
\end{proof}

Lemma~\ref{the:key} also allows to make a stronger conclusion on the precise evolution of the JEP
started at an arbitrary nonrandom initial state. Consider a JEP started at a nonrandom initial
state $A$ with $n$ elements. We will denote by $A_k$ the set obtained by deleting $k$ smallest
elements from $A$, and by $\tau_k$ the time instant when the $k$-th smallest particle from the
initial configuration has just jumped, so that $\tau_1 = \min A + 1$ and $\tau_k = \min A_{k-1} +
1$ for $k = 2,\dots,n$. The following result shows that the memoryless JEP reaches its equilibrium
at the nonrandom time instant $\tau_n$ when all particles have jumped once.

\begin{theorem}
\label{the:Convergence}
Let $(X_t)_{t\in\Z_+}$ be a JEP generated by memoryless height
distributions~\eqref{eq:MemorylessDistribution} and started at a nonrandom initial state $A \in
S_n$. The distribution $X_t$ at any time $t$ can be represented by
\[
 X_t \eqst \left\{
 \begin{aligned}
   A - t,                         & \quad \text{for $t < \tau_1$}, \\
   U(A_k - t, \xi_1,\dots,\xi_k), & \quad \text{if $\tau_k \le t < \tau_{k+1}$ for some $1 \le k \le
   n-1$},\\
   U(\xi_1,\dots,\xi_n),          & \quad \text{for $t \ge \tau_n$}.
 \end{aligned}
 \right.
\]
\end{theorem}
\begin{proof}
First, $X_t = A - t$ for $t < \tau_1$ by the definition of $\tau_1$. When $t = \tau_1$, we see
using the representation~\eqref{eq:NextStep} that $X_t = U(A^*-t, \xi_t) \eqst U(A_1 - t, \xi_1)$.
Thereafter, by repeatedly applying Lemma~\ref{the:key}, we have $X_t = U(A_1 - t, \xi_1)$ for all
$\tau_1 \le t < \tau_2$. Note that $0 \in A_1 - t$ for $t = \tau_2-1$. Therefore, again by
Lemma~\ref{the:key}, $X_t \eqst U(A_2 - t, \xi_1, \xi_2)$ for $t = \tau_2$. By continuing the same
way, we conclude that
\[
 X_t \eqst U(A_k - t, \xi_1,\dots,\xi_k)
\]
for all $t$ such that $\tau_k \le t < \tau_{k+1}$ for some $1 \le k \le n-1$. Finally, at the time
instant $t = \tau_{n} - 1$, $X_t \eqst U(A_{n-1}-t, \xi_1,\dots,\xi_{n-1})$, where the set
$A_{n-1} - t$ only contains the zero element. By one more iteration of Lemma~\ref{the:key}, we
conclude that $X_{\tau_n} \eqst U(\xi_1,\dots,\xi_n)$. This distribution remains invariant by
Theorem~\ref{the:GeometricEquilirium}.
\end{proof}

\subsection{Distribution of the equilibrium}
\label{sec:Entropy}

Having found a stochastic representation for the equilibrium of the JEP with memoryless height
distributions as the noncolliding union of $n$ independent geometric random variables, it is
natural to ask whether there exists a nice formula for the distribution of this random set. The
following result provides an affirmative answer.

\begin{theorem}
\label{the:GibbsRealization}
The random set $G_n = U(\xi_1,\dots,\xi_n)$ defined by~\eqref{eq:RandomSetEmpty} is distributed
according to the Gibbs measure
\begin{equation}
 \label{eq:Gibbs}
 \mu_\beta(B) = Z^{-1}_{n,\beta} \, e^{-\beta H(B)}, \quad B \in S_n,
\end{equation}
where $\beta = -\log \alpha$, the potential energy of $B$ is given by $H(B) = \sum_{x \in B} x$,
and the normalizing constant (a.k.a.\ partition function) is given by
\begin{equation}
 \label{eq:PartitionFunction}
 Z_{n,\beta} = \prod_{k=1}^n \frac{e^\beta}{e^{\beta k} - 1}.
\end{equation}
\end{theorem}
\begin{proof}
The claim follows by substituting $A = \emptyset$ into the more general
Theorem~\ref{the:GibbsARealization} below.
\end{proof}

Consequently, the memoryless JEP attains maximum entropy among all $n$-element random sets with
bounded expected sample mean (e.g.\ Cover and Thomas \cite[Sec.~12.1]{Cover_Thomas_2006}). As a
byproduct, Theorem~\ref{the:GibbsRealization} also provides a computationally efficient way to
generate samples from distribution~\eqref{eq:Gibbs}. A direct computation using~\eqref{eq:Gibbs}
shows that for the memoryless JEP in equilibrium, the expected sample mean height of the particles
is given by
\[
 \E \frac{H(G_n)}{n} = \left( n^{-1} \sum_{k=1}^n \frac{k}{1-\alpha^k} \right) - 1,
\]
the probability of the ground state by
\[
 \pr(G_n = [0,n-1])
 = \alpha^{n(n+1)/2} \prod_{k=1}^n \left( \frac{1-\alpha^k}{\alpha^k} \right),
\]
and the probability that zero is occupied by
\[
 \pr(0 \in G_n) = 1 - \alpha^n.
\]
By Birkhoff's ergodic theorem, the last formula is also equal to the long-run average jump rate in
the system---a simple way to prove the formula is to observe that the minimal element of $G_n$ is
geometrically distributed with parameter $\alpha^n$ due to Lemma~\ref{the:UnionMin}.

We will next state and prove a generalization of Theorem~\ref{the:GibbsRealization}, which allows
to compute the nonequilbrium distributions of a memoryless JEP started at an arbitrary nonrandom
initial state, as described in Theorem~\ref{the:Convergence}. Denote the family of supersets of a
finite (possibly empty) set $A$ with $|A|+n$ elements by
\[
 S_{A,n} = \left\{ B \subset \Z_+: \ B \supset A, \ |B \setminus A| = n \right\}.
\]

\begin{theorem}
\label{the:GibbsARealization}
The random set $G_n(A) = U(A,\xi_1,\dots,\xi_n)$ defined by~\eqref{eq:RandomSet} is distributed
according to the Gibbs measure
\begin{equation}
 \label{eq:GibbsSuper}
 \mu_{A,\beta}(B) = Z^{-1}_{n,\beta} \, e^{-\beta H_A(B)}, \quad B \in S_{A,n},
\end{equation}
where $\beta = -\log \alpha$, $Z_{n,\beta}$ is the constant defined
by~\eqref{eq:PartitionFunction}, and $H_A(B) = \sum_{x \in B \setminus A} h_A(x)$ is the
\emph{$A$-neglecting potential energy} of $B$, where
\begin{equation}
 \label{eq:EnergyA}
 h_A(x) = |[0,x-1] \setminus A|
\end{equation}
is the number of elements in $A^c$ strictly less than $x$.
\end{theorem}
\begin{proof}
Consider first the case $n=1$. If $B = A \cup \{x\}$ for some $x \in A^c$, then $G_1(A) = B$ if
and only if $\theta_A(\xi_1) = x$. The probability of this event is $(1-\alpha) \alpha^{h_A(x)}$
by Lemma~\ref{the:ShiftDistribution}. Because $H_A(B) = h_A(x)$ for such $B$, the claim follows
for $n=1$.

Assume next that the claim is true for some $n \ge 1$. Fix a set $B \supset A$ such that
$|B\setminus A| = n+1$. Then
\begin{align*}
 \pr(G_{n+1}(A) = B)
 &= \sum_{x \in B \setminus A} \pr(G_n(A) = B \setminus \{x\}, \ \theta_{B \setminus \{x\}}(\xi_{n+1}) = x)\\
 &= \sum_{x \in B \setminus A} \pr(G_n(A) = B \setminus \{x\}) \pr( \theta_{B \setminus \{x\}}(\xi_{n+1}) = x).
\end{align*}
Because $H_A(B\setminus \{x\}) = H_A(B) - h_A(x)$ for all $x \in A^c$, it follows by the induction
assumption that
\[
 \pr(G_n(A) = B \setminus \{x\}) = Z_{n,\beta}^{-1} \, \alpha^{H_A(B) - h_A(x)}.
\]
Further,
\[
 \pr( \theta_{B \setminus \{x\}}(\xi_{n+1}) = x)
 \, = \, (1-\alpha) \alpha^{h_{B \setminus \{x\}}(x)}
\]
by Lemma~\ref{the:ShiftDistribution}. Note also that $h_{B \setminus \{x\}}(x) = h_B(x)$ and
\[
 h_A(x) = h_B(x) + | [0,x-1] \cap (B\setminus A) |.
\]
By combining these observations, we find that
\begin{equation}
 \label{eq:ShiftDistribution1}
 \pr(G_{n+1}(A) = B)
 = (1-\alpha) Z_{n,\beta}^{-1} \alpha^{H_A(B)} \sum_{x \in B \setminus A} \alpha^{-| [0,x-1] \cap (B\setminus A)
 |}.
\end{equation}
To compute the last sum, observe that $| [0,x-1] \cap (B\setminus A) |$ equals $k-1$ when $x$ is
the $k$-th smallest element in $B \setminus A$. Therefore,
\begin{equation}
 \label{eq:ShiftDistribution2}
 \sum_{x \in B \setminus A} \alpha^{-| [0,x-1] \cap (B\setminus A) |}
 = \sum_{k=1}^{n+1} \alpha^{-(k-1)}
 = \frac{1-\alpha^{-n-1}}{1-\alpha^{-1}}.
\end{equation}
By combining~\eqref{eq:ShiftDistribution1} and~\eqref{eq:ShiftDistribution2}, we conclude that the
claim is true for $n+1$.
\end{proof}

\section{Related models}
\label{sec:RelatedModels}

\subsection{Bounded JEP with uniform jumps}
\label{subsec:finite}
Consider the $n$-particle JEP where jump heights are bounded by a constant $M \ge n$, and the
particles jump into unoccupied sites in $[0,M-1]$ uniformly randomly. In this case Warrington
\cite{Warrington_2005} has shown that the equilibrium distribution is given by
\begin{equation}
 \label{eq:Warrington}
 \mu(B) = Z_{n,M}^{-1} \prod_{x \in B} \big( 1 + | [x+1,M-1] \setminus B | \big),
\end{equation}
where the normalization constant $Z_{n,M} = \left\{
\begin{smallmatrix}M+1\\M+1-n\end{smallmatrix} \right\}$ is a Stirling number of the second kind
(Stanley \cite[Corollary 2.4.2]{Stanley_1997}). This equilibrium may also be expressed as the
Gibbs measure
\[
 \mu(B) = Z_{n,M}^{-1} e^{- \beta H_M(B)},
\]
where $\beta = 1$ and
\[
 H_M(B) = - \sum_{x \in B} \log \big( 1 + | [x+1,M-1] \setminus B | \big).
\]
Recall that the memoryless JEP reaches its equilibrium when the initially highest particle jumps
for the first time. The same is not true for the bounded JEP with uniform jumps. To see this in
the one-particle case, observe that for $n=1$ the equilibrium distribution~\eqref{eq:Warrington}
reduces to
\[
 \mu(\{x\}) = Z_{1,M}^{-1}(M-x), \quad x \in [0,M-1],
\]
whereas just after each jump, the particle is uniformly distributed in $[0,M-1]$.

\subsection{Finite ASEP with a reflecting boundary}

Consider a continuous-time asymmetric simple exclusion process on $\Z_+$ consisting of $n$
identical particles, where each particle attempts a jump $x \mapsto x+1$ at rate $\lambda$ and a
jump $x \mapsto x-1$ at rate $\eta$, independently of the other particles (e.g.\
Grimmett~\cite[Sec.~10.4]{Grimmett_2010}). An attempted jump takes place if the targeted site is
unoccupied and belongs to $\Z_+$; otherwise the attempt is suppressed.

Assume that $0 < \lambda < \eta$, and let $Y$ be the Markov jump process in the Weyl chamber $W_n
= \{x \in \Z_+^n: x_1 < x_2 < \cdots < x_n\}$ which keeps track of the particle locations. If
$\tilde Y$ denotes the corresponding free process in $\Z_+^n$ where particles are allowed to jump
on top of each other, then the components of $\tilde Y$ are independent $M/M/1$ queues, and so the
equilibrium distribution of $\tilde Y$ is a product of $n$ independent geometric distributions on
$\Z_+$ with success probability $\lambda/\eta$. Further, because $\tilde Y$ is reversible, we
obtain the equilibrium of $Y$ be truncating the equilibrium of $\tilde Y$ to the Weyl chamber
$W_n$ (e.g.\ Kelly~\cite[Cor.~1.10]{Kelly_1979}). As a consequence, we see that the equilibrium of
the corresponding set-valued process in $S_n$ is the Gibbs measure~\eqref{eq:Gibbs} with $\beta =
-\log (\lambda/\eta)$.

Thus, the memoryless JEP and the finite ASEP with a reflecting boundary have the same equilibrium
distribution.

\section{Conclusions}
\label{sec:Conclusions}

This paper introduced a juggler's exclusion process (JEP) as a special discrete-time exclusion
process on the positive integers, where the motion of particles is deterministic outside the
boundary, and particles hitting the boundary jump into random unoccupied sites. We showed that a
JEP is ergodic if the jump heights are uniformly integrable and satisfy a natural aperiodicity
condition. The main part of the analysis was devoted to JEPs where the jump distributions satisfy
a set-avoiding memoryless property. Such JEPs were shown to converge to equilibrium in finite
nonrandom time, where the equilibrium is a Gibbs measure corresponding to a linear gravitational
potential. The proof of this result was based on representing the time-dependent system
configuration in terms of independent geometric random variables. As a byproduct, the proof
yielded a fast computational way to generate samples from such Gibbs measures.

There are two open problems related to this model that we find particularly interesting. First, do
we obtain an ergodic JEP if we assume that jump distributions are bounded in $L^1$ but not
uniformly integrable? Second, is it possible to write down an analytical formula for the
equilibrium of a general ergodic JEP using techniques from renewal theory, as in the special case
of one-particle JEP? The model analyzed in this paper also leaves room for many generalizations.
For example, one might consider analogous models with a continuous time parameter and fully random
particle motion as in ASEP models. Another interesting direction is to look for continuous-time
processes where particles move in continuum and jump into random locations not too close from the
other particles.

\section*{Acknowledgements}

We thank Venkat Anantharam, Matti Vihola, and Greg Warrington for helpful and inspiring
discussions. L.~Leskelä has been partially funded by the Academy of Finland.

\appendix

\section{Technical details}
\label{sec:Technical}

\subsection{Set-avoiding shift}
\label{sec:ShiftTechnical}

Recall that $\theta_A$ is the $A$-avoiding shift defined by~\eqref{eq:DefShift} in
Section~\ref{sec:MemorylessDistribution}.

\begin{lemma}
\label{the:IndexShift}
For all finite $A \subset \Z_+$ and all $x \in \Z_+$,
\[
 \theta_{A+1}(x+1) = \theta_A(x) + 1.
\]
\end{lemma}
\begin{proof}
Because the equations
\begin{align*}
& \min \{n \ge 0: |\,[0,n]\setminus A\,| \ge x+1 \} \\
& = \min \{n \ge 0: |\,[0,n+1]\setminus A\,| \ge x+1 \} + 1
\end{align*}
and
\[
 |\, [0,n+1] \setminus (A+1) \,| = |\, [0,n]\setminus A| + 1
\]
hold for all $A \subset \Z_+$ and all $x \in \Z_+$, we obtain
\begin{align*}
 \theta_{A+1}(x+1)
 &= \min \{n \ge 0: |\,[0,n]\setminus (A+1)\,| \ge x+2 \} \\
 &= \min \{n \ge 0: |\,[0,n+1]\setminus (A+1)\,| \ge x+2 \} + 1 \\
 &= \min \{n \ge 0: |\,[0,n]\setminus A\,| \ge x+1 \} + 1 \\
 &= \theta_A(x) + 1.
\end{align*}
\end{proof}

\begin{lemma}
\label{the:IndexDelete}
For any finite nonempty $A \subset \Z_+$ and any $x \in \Z_+$,
\[
  \theta_{A}(x) = \left\{
  \begin{aligned}
  x                , &\quad x < \min A, \\
  \theta_{A^*}(x+1), &\quad x \ge \min A.
  \end{aligned}
  \right.
\]
\end{lemma}
\begin{proof}
The first claim is that
\[
\min\{n\ge 0: |\,[0,n]\setminus A\,| \ge x + 1\} = x
\]
whenever $[0,x]\setminus A = [0,x]$. Because
\[
|\,[0,n] \setminus A\,| \ge x + 1 = |\,[0,x]\,| = |\,[0,x]\setminus A\,|
\]
implies that $n \ge x$, the claim follows.

The second claim is that
\[
\min\{n\ge 0: |\,[0,n]\setminus A\,| \ge x + 1\} = \min\{n\ge 0: |\,[0,n]\setminus A^*\,| \ge x +
2\}
\]
whenever $x \ge \min A$. Note that $|\,[0,y] \setminus A\,| = |\,[0,y] \setminus A^*\,|-1$
whenever $y \ge \min A$. If
\[
y \in \{n\ge 0: |\,[0,n]\setminus A\,| \ge x + 1\},
\]
then $|\,[0,y] \setminus A\,| \ge x+1 = |\,[0,x]\,| \ge |\,[0,x]\setminus A\,|$, so $y \ge x \ge
\min A$. By the note above, $|\,[0,y]\setminus A\,| = |\,[0,y]\setminus A^*\,| -1$, i.e.
\[
y \in \{n\ge 0: |\,[0,n]\setminus A^*\,| \ge x + 2\}.
\]
The other direction is similar, and the sets to be minimized are the same.
\end{proof}

\begin{lemma}
\label{the:ShiftInverse}
For any finite $A \subset \Z_+$ and any $x \in A^c$,
\[
 \theta_A^{-1}(x) = h_A(x),
\]
where $h_A(x) = | \, [0,x-1] \setminus A \, |$.
\end{lemma}
\begin{proof}
Now the claim is that
\[
\min\{n\ge 0: |[0,n]\setminus A| \ge |[0,x-1]\setminus A| + 1\} = x
\]
whenever $x \notin A$. In this case $|\,[0,x]\setminus A\,| = |\,[0,x-1]\setminus A\,|+1$, so the
claim reduces to
\[
\min\{n\ge 0: |\,[0,n]\setminus A\,| \ge |\,[0,x]\setminus A\,|\} = x,
\]
which is trivial.
\end{proof}

\subsection{Noncolliding union}
\label{sec:UnionTechnical}

In this section we summarize basic technical properties of the noncolliding union map
$U(A,x_1,\dots,x_n)$ defined by~\eqref{eq:DefUnion} in Section~\ref{sec:Union}.

\begin{lemma}
\label{the:UnionMin}
For all finite $A \subset \Z_+$ and all $x_1,\dots,x_n \in \Z_+$,
\begin{align*}
 \min U(A,x_1,\dots,x_n)  &= \min( A \cup \{x_1,\dots,x_n\} ).
\end{align*}
\end{lemma}
\begin{proof}
Let us first prove the claim for $n=1$. We may assume that $A$ is nonempty, because otherwise the
claim is trivial. If $x_1 < \min(A)$, then $\theta_A(x_1) = x_1$, which implies that $\min
U(A,x_1) = x_1$. If $x_1 \ge \min(A)$, then $\theta_A(x_1) > \min(A)$, which implies that $\min
U(A,x_1) = \min(A)$. Therefore, we conclude that the claim holds for $n=1$. The general case
follows by induction, due to the recursive construction of $U(A,x_1,\dots,x_n)$.
\end{proof}

\begin{lemma}
\label{the:UnionShift}
For all finite $A \subset \Z_+$ and all $x_1,\dots,x_n \in \Z_+$,
\[
 U(A+1,x_1+1,\dots,x_n+1) = U(A,x_1,\dots,x_n) + 1.
\]
\end{lemma}
\begin{proof}
Note first that by Lemma~\ref{the:IndexShift},
\begin{align*}
 U(B+1,y+1)
 &= (B+1) \cup \{\theta_{B+1}(y+1)\} \\
 &= (B+1) \cup \{\theta_{B}(y) + 1\} \\
 &= U(B, y) + 1
\end{align*}
for all finite $B \subset \Z_+$ and all $y \in \Z_+$. Therefore, the claim is true whenever $n=1$.
If the claim holds for some $n \ge 1$, then by the induction assumption
\begin{align*}
 U(A+1,x_1+1,\dots,x_n+1)
 &= U( U(A+1,x_1+1,\dots,x_n+1), x_{n+1}+1) \\
 &= U( B + 1, x_{n+1}+1),
\end{align*}
where $B = U(A,x_1,\dots,x_n)$. The claim now follows by applying the property with $n=1$.
\end{proof}

\begin{lemma}
\label{the:UDelete}
For any finite nonempty $A \subset \Z_+$ and any $x \in \Z_+$,
\[
 U(A,x)^* = \left\{
 \begin{aligned}
   A,          & \quad x < \min A, \\
   U(A^*,x+1), & \quad x \ge \min A.
 \end{aligned}
 \right.
\]
\end{lemma}
\begin{proof}
First, if $x < \min A$, then $\theta_A(x) = x$ by Lemma~\ref{the:IndexDelete}. Therefore,
$U(A,x)^* = (A \cup \{x\})^* = A$. Next, if $x \ge \min A$, then $\theta_A(x) = \theta_{A^*}(x+1)$
by Lemma~\ref{the:IndexDelete}, and $\min U(A,x) = \min A$. It follows that
\[
 U(A,x)^*
 = A^* \cup \{\theta_A(x)\}
 = A^* \cup \{\theta_{A^*}(x+1)\}
 = U(A^*, x+1).
\]
\end{proof}

\newcommand{\SortNoop}[1]{}\def\cprime{$'$}

\end{document}